\documentclass[a4,12pt]{amsart}
\usepackage{amssymb,amstext,amsmath,amscd,amsthm,amsfonts,enumerate,latexsym}
\usepackage{color}
\usepackage[dvipdfmx]{graphicx}
\usepackage[all]{xy}
\usepackage{extarrows}

\usepackage{ytableau}

\theoremstyle{plain}
\newtheorem{thm}{Theorem}[section]
\newtheorem*{thm*}{Theorem}
\newtheorem*{cor*}{Corollary}

\newtheorem{prop}[thm]{Proposition}
\newtheorem{lem}[thm]{Lemma}
\newtheorem{cor}[thm]{Corollary}
\newtheorem{claim}{Claim}
\newtheorem*{claim*}{Claim}

\theoremstyle{definition}
\newtheorem{defn}[thm]{Definition}
\newtheorem{ex}[thm]{Example}

\newtheorem{setting}[thm]{Setting}

\newtheorem{rem}[thm]{Remark}

\theoremstyle{remark}

\newtheorem*{ac}{Acknowledgments}

\numberwithin{equation}{thm}

\def\rank{\mathrm{rank}}

\def\e{\mathrm{e}}
\def\m{\mathfrak m}
\def\n{\mathfrak n}

\def\q{\mathfrak q}

\def\K{\mathrm{K}}

\newcommand{\rma}{\mathrm{a}}

\newcommand{\rme}{\mathrm{e}}

\newcommand{\rmr}{\mathrm{r}}

\newcommand{\rmI}{\mathrm{I}}

\newcommand{\rmK}{\mathrm{K}}

\newcommand{\calR}{\mathcal{R}}

\newcommand{\fkm}{\mathfrak{m}}

\newcommand{\mapright}[1]{%
\smash{\mathop{%
\hbox to 1cm{\rightarrowfill}}\limits^{#1}}}

\newcommand{\mapleft}[1]{%
\smash{\mathop{%
\hbox to 1cm{\leftarrowfill}}\limits_{#1}}}

\begin{document}

\setlength{\baselineskip}{14pt}
\title[On the almost Gorenstein property of determinantal rings]{On the almost Gorenstein property of \\ determinantal rings}
\author{Naoki Taniguchi}
\address{Department of Mathematics, School of Science and Technology, Meiji University, 1-1-1 Higashi-mita, Tama-ku, Kawasaki 214-8571, Japan}
\email{taniguchi@meiji.ac.jp}
\urladdr{http://www.isc.meiji.ac.jp/~taniguchi/}

\thanks{2010 {\em Mathematics Subject Classification.} 13H10, 13H15, 13A02}
\thanks{{\em Key words and phrases.} almost Gorenstein local ring, almost Gorenstein graded ring, determinantal ring}
\thanks{The author was partially supported by JSPS Grant-in-Aid for Scientific Research 26400054.}

\begin{abstract} 
In this paper we investigate the question of when the determinantal ring $R$ over a field $k$ is an almost Gorenstein local/graded ring in the sense of \cite{GTT}. As a consequence of the main result, we see that if $R$ is a non-Gorenstein almost Gorenstein local/graded ring, then the ring $R$ has a minimal multiplicity.
\end{abstract}


\maketitle


\section{Introduction}\label{intro}
Let $k$ be an infinite field and $m$, $n \ge 2$ be integers. Let $X = \left[X_{ij}\right]$ be an $m\times n$ matrix of indeterminates over the field $k$. We denote by $S=k[X]$ the polynomial ring generated by $\{X_{ij}\}_{1 \le i \le m, ~1\le j \le n}$ over $k$ and consider $S$ as a $\Bbb Z$-graded ring under the grading $S_0 = k,~ X_{ij} \in S_1$. Let $\rmI_t(X)$ be the ideal of $S$ generated by the $t \times t$-minors of the matrix $X$, where $2 \le t \le \min\{m, n\}$. We put $R = S/{\rm I}_t(X)$ which is called {\it the determinantal ring}. The result of M. Hochster and J. A. Eagon (\cite{HE}) insists that the ring $R$ is always a Cohen-Macaulay normal integral domain of dimension $mn-(m-(t-1))(n-(t-1))$. Moreover the Gorensteinness of the determinantal ring is characterized by $m=n$ (\cite{Svanes}).


Almost Gorenstein rings are one of the candidates for a new class of Cohen-Macaulay rings which may not be Gorenstein, but sufficiently good next to the Gorenstein rings. The concept of this kind of local rings dates back to the paper \cite{BF} given by V. Barucci and R. Fr\"oberg in $1997$ where the base local ring is analytically unramified of dimension one. However, since the notion given by \cite{BF} was not flexible for the analysis of analytically ramified case, so that in 2013 S. Goto, N. Matsuoka and T. T. Phuong \cite{GMP} proposed the notion over one-dimensional Cohen-Macaulay local rings, using the behaviors of the first Hilbert coefficient of canonical ideals. Finally in 2015 S. Goto, R. Takahashi and the author \cite{GTT} gave the definition of almost Gorenstein {\it graded/local} rings of arbitrary dimension in order to use the theory of Ulrich modules. It is proved by \cite{GTT} that every one-dimensional Cohen-Macaulay local ring of finite Cohen-Macaulay representation type and two-dimensional rational singularity are almost Gorenstein local rings. In addition non-Gorenstein almost Gorenstein rings have G-regularity in the sense of \cite{greg}, that is, every totally reflexive module is free. Let us now remark that even if $A_N$ is an almost Gorenstein local ring, $A$ is not necessarily an almost Gorenstein graded ring, where $N$ denotes the unique graded maximal ideal of a Cohen-Macaulay graded ring $A$ (see \cite[Theorems 2.7, 2.8]{GMTY2}, \cite[Example 8.8]{GTT}).

The purpose of the present paper is to study the question of when the determinantal rings are almost Gorenstein rings. 
Throughout this paper, unless otherwise specified we assume $m \le n$, because we may replace $X$ by its transpose if necessary. Let $M=R_+$ stand for the graded maximal ideal of $R$.

With this notation the main result of this paper is stated as follows. 

\begin{thm}\label{main}
The following conditions are equivalent.
\begin{enumerate}
\item[$(1)$] $R$ is an almost Gorenstein graded ring.
\item[$(2)$] $R_M$ is an almost Gorenstein local ring.
\item[$(3)$] Either $m=n$, or $m\ne n$ and $m=t=2$.
\end{enumerate}
\end{thm}

As a consequence of Theorem \ref{main}, the almost Gorenstein property for the local ring $R_M$ implies that the ring has a minimal multiplicity, provided $m \ne n$ (see Remark \ref{5.3}). 

Furthermore Theorem \ref{main} yields the following application. In the case where the field $k$ has a characteristic $0$, then the determinantal ring appears as the ring of invariants. 
More precisely let $Y$ (resp. $Z$) be an $m \times (t-1)$ matrix (resp. a $(t-1) \times n$ matrix) of indeterminates over $k$. We put $A=k[Y, Z]$ and $G={\rm GL}_{t-1}(k)$ the general linear group. Suppose that the group $G$ acts on the ring $A$ as $k$-automorphisms by taking $Y$ (resp. $Z$) onto $YT^{-1}$ (resp. $TZ$) for every $T \in G$. Then the classical result of C. D. Concini and C. Procesi (\cite{CP}) shows that the ring $A^G$ of invariants is generated by the entries of the $m \times n$ matrix $X=YZ$ and the ideal of relations on $X$ is generated by the $t \times t$-minors of $X$ (see also \cite[Theorem (7.6)]{BV}). Hence Theorem \ref{main} induces the following invariant-theoretic result.

\begin{cor}
Let $A$ and $G$ be as above. Then $A^G$ is an almost Gorenstein graded ring if and only if either $m=n$, or $m\ne n$ and $m=t=2$.
\end{cor}

Let us now explain how this paper is organized. In Section $2$ we shall give fundamental properties on almost Gorenstein rings, including the definition in the sense of \cite{GTT}. The purpose of Section $3$ is to give a proof of Theorem \ref{main} in the case where $t=2$. In Section $4$ we will recall the resolution of determinantal rings due to P. Roberts in order to determine the lower bound of the number of generators for $M \rmK_R$, where $\rmK_R$ denotes the graded canonical module of $R$. Finally we shall prove Theorem \ref{main} in Section $5$.


\section{Preliminaries}

The aim of this section is mainly to summarize some basic results on almost Gorenstein rings, which we will use throughout this paper. In what follows, let $(R, \m)$ be a Cohen-Macaulay local ring with $d=\dim R$ possessing the canonical module $\rmK_R$. Then it is well-known that $R$ is a homomorphic image of a Gorenstein ring. 

\begin{defn}\label{2.1}(\cite[Definition 3.3]{GTT})
We say that $R$ is {\it an almost Gorenstein local ring}, if there exists an exact sequence
$$
0 \to R \to \rmK_R \to C \to 0
$$
of $R$-modules such that $\mu_R(C) = \rme^0_{\m}(C)$, where $\mu_R(C)$ denotes the number of elements in a minimal system of generators for $C$ and
$$
\rme^0_{\m}(C) = \lim_{n \to \infty}\frac{\ell_R(C/\m^{n+1}C)}{n^{d-1}}\cdot (d-1)!
$$ 
is the multiplicity of $C$ with respect to $\fkm$.
\end{defn}

Note that every Gorenstein ring is an almost Gorenstein local ring, and the converse holds if the ring $R$ is Artinian (\cite[Lemma 3.1 (3)]{GTT}). Definition \ref{2.1} requires  that if $R$ is an almost Gorenstein local ring, then $R$ might not be Gorenstein but the ring $R$ can be embedded into its canonical module $\K_R$ so that the difference $\K_R/R$ should have good properties.
For any exact sequence
$$
0 \to R \to \rmK_R \to C \to 0
$$
of $R$-modules, we notice that $C$ is a Cohen-Macaulay $R$-module of dimension $d-1$, if $C \ne (0)$ (\cite[Lemma 3.1 (2)]{GTT}).
Suppose that $R$ possesses an infinite residue class field $R/ \m$. Set $\overline{R}=R/[(0):_RC]$ and let $\overline{\m}$ denote the maximal ideal of $\overline{R}$. 
Choose elements $f_1, f_2, \ldots, f_{d-1} \in \m$ such that $(f_1, f_2, \ldots, f_{d-1})\overline{R}$ forms a minimal reduction of $\overline{\m}$. Then we have
$$
\rme_{\m}^0(C) = \rme_{\overline{\m}}^0(C) = \ell_R(C/(f_1, f_2, \ldots, f_{d-1})C) \ge \ell_R(C/\m C) = \mu_R(C). 
$$
Therefore $\rme_{\m}^0(C) \ge \mu_R(C)$ and we say that $C$ is {\it an Ulrich $R$-module} if $\rme_{\m}^0(C) = \mu_R(C)$, since $C$ is {\it a maximally generated maximal Cohen-Macaulay $\overline{R}$-module} in the sense of \cite{BHU}. Thus $C$ is an Ulrich $R$-module if and only if $\m C = (f_1, f_2, \ldots, f_{d-1})C$. Therefore if $\dim R=1$, then the Ulrich property for $C$ is equivalent to saying that $C$ is a vector space over $R/\m$.

One can construct many examples of almost Gorenstein rings (e.g., \cite{GMTY1, GMTY2, GMTY3, GMTY4, GTT}). The significant examples of almost Gorenstein rings are one-dimensional Cohen-Macaulay local rings of finite Cohen-Macaulay representation type and two-dimensional rational singularities (\cite[Corollary 11.4, Theorem 12.1]{GTT}). The origin of the theory of almost Gorenstein rings are the theory of numerical semigroup rings, so that there are a lot of examples of almost Gorenstein numerical semigroup rings (see \cite{BF, GMP}), also the corresponding semigroup is called {\it almost symmetric}.

Let us begin with the fundamental result on almost Gorenstein rings.

\begin{lem}\label{2.2}$($\cite[Corollary 3.10]{GTT}$)$
Let $R$ be an almost Gorenstein local ring and choose an exact sequence  
$$
0 \to R \overset{\varphi}{\longrightarrow} \rmK_R \to C \to 0
$$
of $R$-modules such that $\mu_R(C) = \rme_\fkm^0(C)$. If $\varphi (1) \in \fkm\rmK_R$, then $R$ is a regular local ring. 
Therefore $\mu_R(C) = \rmr (R) - 1$ if $R$ is not regular.
\end{lem}

We apply this lemma to immediately get the following corollary.

\begin{cor}\label{2.3}
Let $R$ be an almost Gorenstein local ring but not Gorenstein. Choose an exact sequence
$$
0 \to R \overset{\varphi}{\longrightarrow} \rmK_R \to C \to 0
$$
of $R$-modules such that $C$ is an Ulrich $R$-module. 
Then 
$$
0 \to \m \varphi(1) \to \m \rmK_R \to \m C \to 0
$$
is an exact sequence of $R$-modules.

When this is the case, we have an inequality 
$$
\mu_R(\m \rmK_R) \le \mu_R(\m) + \mu_R(\m C).
$$
\end{cor}


Let us discuss the case of graded rings. Let $R=\bigoplus_{n\ge 0}R_n$ be a Cohen-Macaulay graded ring and assume that $R_0$ is a local ring and there exists the graded canonical module $\rmK_R$. Let $a=\rma (R)$ stand for an a-invariant of $R$.

\begin{defn}(\cite[Definition 8.1]{GTT})
We say that $R$ is {\it an almost Gorenstein graded ring}, if there exists an exact sequence
$$0 \to R \to \mathrm{K}_R(-a) \to C \to 0$$
of graded $R$-modules such that $\mu_R(C) = \e_M^0(C)$, where $M$ denotes the unique graded maximal ideal of $R$. Remember that $\mathrm{K}_R(-a)$ stands for  the graded $R$-module whose underlying $R$-module is the same as that of $\K_R$ and whose grading is given by $[\mathrm{K}_R(-a)]_n = [\mathrm{K}_R]_{n-a}$ for all $n \in \Bbb Z$.
\end{defn}

Notice that every Gorenstein graded ring is by definition an almost Gorenstein graded ring. Moreover the local ring $R_{M}$ is almost Gorenstein if $R$ is an almost Gorenstein graded ring, because $C_M$ is an Ulrich $R_M$-module and $\mathrm{K}_{R_M} \cong [\mathrm{K}_R]_M$. Unfortunately, the converse is not true in general (e.g., \cite[Theorems 2.7, 2.8]{GMTY2}, \cite[Example 8.8]{GTT}).


For the rest of this section, let $S=k[X_1, X_2, \ldots, X_n]$ denote the polynomial ring over a field $k$. We consider $S$ as a $\Bbb Z$-graded ring with $S_0=k$, $\deg X_i=1$ for each $1 \le i\le n$. Let $K$ be a finitely generated graded $S$-module and assume that $K$ has a following presentation 
$$
F_1:=S^{\oplus a}(-1) \overset{\varphi}{\longrightarrow} F_0:=S^{\oplus b}\overset{\varepsilon}{\longrightarrow} K \longrightarrow 0 \ \ \ \ \ (\sharp)
$$
of graded $S$-modules, where $a, b$ are positive integers.
Let $M=S_+$ stand for the unique graded maximal ideal of $S$. Then $\varphi$ induces the graded $S$-linear homomorphism 
$$
\overline{\varphi} : (F_1)_+/M(F_1)_+ \longrightarrow  (F_0)_+/M(F_0)_+.
$$

With this notation we have the following, where $V$ denotes the image of $\overline{\varphi}$.

\begin{prop}\label{2.4}
There is the equality
$$
\mu_S(MK) = nb - \dim_k V.
$$
In particular, $\mu_S(MK) \ge nb -a$.
\end{prop}

\begin{proof}
Notice that we have the exact sequence
$$
(F_1)_+ \to (F_0)_+ \to K_+ \to 0
$$
of graded $S$-modules. Applying the functor $S/M \otimes_S -$, we get 
$$
k^{\oplus a} \cong (F_1)_+/M (F_1)_+ \overset{\overline{\varphi}}{\longrightarrow} (F_0)_+/M(F_0)_+ \longrightarrow K_+/MK_+ \longrightarrow 0
$$
since $(F_1)_+ = S^{\oplus a}$.
Therefore
$$
0 \to V \to MF_0/M^2F_0 \to MK/M^2K \to 0
$$
yields that $\mu_S(MK) = n b -\dim_k V$
as desired.
\end{proof}

\begin{cor}
Let $m \ge 2$ be an integer. Let $X=[X_{ij}]$ be an $m \times (m+1)$ matrix of indeterminates over a field $k$, $S=k[X]$ be the polynomial ring generated by $\{X_{ij}\}_{1 \le i \le m, 1 \le j \le m+1}$ over $k$. We put $R=S/{\rm I}_m(X)$ and $M = R_+$.
Then 
$$
\mu_R(M \rmK_R) = (m^2 -1)(m+1).
$$
\end{cor}

\begin{proof}
The Eagon-Northcott resolution (\cite{EN}) implies the presentation of the graded canonical module $\rmK_R$ of $R$
$$
S^{\oplus (m+1)}(-m^2) \overset{X}{\longrightarrow} S^{\oplus m}(-m^2+1) \longrightarrow \rmK_R \longrightarrow 0
$$
as graded $S$-modules.
Since $V$ is generated by the columns of the matrix $X$, we then have $\dim_k V = m+1$. Thanks to Proposition \ref{2.4}, we get
$$
\mu_R(M \rmK_R) = m^2(m+1)-(m+1) = (m^2-1)(m+1)
$$
which completes the proof.
\end{proof}


\section{The case where $t=2$}

First of all we fix our notation and assumptions on which all the results in this section are based.

\begin{setting}\label{3.1}
Let $2 \le t \le m \le n$ be integers, $X=[X_{ij}]$ an $m \times n$ matrix of indeterminates over an infinite field $k$. We put $R = S/{\rm I}_t(X)$ and $M=R_+$, where $S=k[X]$ stands for the polynomial ring over the field $k$. Let $Y$ be the matrix obtained from $X$ by choosing the first $t-1$ columns and set $Q = \rmI_{t-1}(Y)R$. We denote by $x_{ij}$ the image of $X_{ij}$ in $R$.
\end{setting}

This section focuses our attention on proving a part of Theorem \ref{main} in the case where $t=2$.

\begin{thm}\label{3.2}
Suppose that $t=2$. If $R_M$ is an almost Gorenstien local ring and $m\ne n$, then $m=2$. 
\end{thm}

To prove Theorem \ref{3.2}, we need some auxiliaries.

\begin{lem}\label{3.3}
Suppose that $t=2$. Then there is the equality
$$
\mu_R(M Q^{\ell}) = \binom{m+\ell}{\ell +1}\cdot n
$$
for every $\ell \ge 0$.
\end{lem}

\begin{proof}
We may assume $\ell >0$. If $n=2$, then $m=n$ and $x_{11}x_{22}=x_{12}x_{21}$ in $R$. Thus $Q^{\ell} = (x_{11}^{\ell}, x_{11}^{\ell-1}x_{21}, \ldots, x_{11}x_{21}^{\ell-1}, x_{21}^{\ell})$, because $\mu_R(Q)=2$. Therefore $\mu_R(M Q^{\ell}) = 2 \ell +4 = 2(\ell + 2)$. Suppose that $n >2$ and the assertion holds for $n-1$. Put 
$$
M' = (x_{ij}\mid 1 \le i \le m, ~ 1 \le j \le n-1) \ \ \text{and} \ \ \q = (x_{in}\mid 1 \le i \le m).
$$ 
Then $M = M' + \q$ and we have the following.

\begin{claim}\label{claim1}
$\mu_R(M Q^{\ell}) = \mu_R(M' Q^{\ell}) + \mu_R(\q Q^{\ell})$.
\end{claim}

\begin{proof}[Proof of Claim \ref{claim1}]
Let $X_1 = [X_{ij}]$ be the matrix obtained from $X$ by removing the $n$-th column. We put $R_1 = k[X_1]/{\rm I}_2(X_1)$. Then the composite map
$$
R_1 = k[X_1]/{\rm I}_2(X_1) \overset{\varphi}{\longrightarrow} R=k[X]/{\rm I}_2(X) \overset{\varepsilon}{\longrightarrow} k[X]/{{\rm I}_2(X) + \q} =:\overline{R}
$$
makes an isomorphism $R_1 \cong \overline{R}$. Set 
$$
M_1 = (x_{ij} \mid 1 \le i \le m, 1 \le j \le n-1) \ \ \text{and} \ \ Q_1 = (x_{i1} \mid 1 \le i \le m)
$$
inside of the ring $R_1$. We then have $\varphi(M_1 Q_1^{\ell}) = M' Q^{\ell}$. Let $\{f_{\alpha}\}$ (resp. $\{g_{\beta}\}$) be a homogeneous minimal system of generators for $M_1 Q_1^{\ell}$ (resp. $\q Q^{\ell}$). In order to see the homogeneous component of degree $\ell +1$,  $\{\varphi(f_{\alpha})\}$ and $\{g_{\beta}\}$ forms a minimal system of generators for $MQ^{\ell}$, so that $\mu_R(M Q^{\ell}) = \mu_R(M' Q^{\ell}) + \mu_R(\q Q^{\ell})$.
\end{proof}

Let $f, g \in Q^{\ell}$ be non-zero monomials in $x_{ij}$ which forms a part of a minimal basis of $Q^{\ell}$. Remember that $x_{i1}x_{jn} = x_{in}x_{j1}$ for every $1 \le i < j \le m$ and $R$ is an integral domain. Thus $x_{i1}f = x_{j1}g$ if and only if $x_{in}f = x_{jn}g$ for every $1 \le i < j \le m$, whence
$$
\mu_R(\q Q^{\ell}) = \mu_R(Q^{\ell+1}) = \binom{m+\ell}{\ell + 1}.
$$
Hence we have the equalities
\begin{eqnarray*}
\mu_R(M Q^{\ell}) &=& \mu_R(M' Q^{\ell}) + \mu_R(\q Q^{\ell}) \\
&=& \binom{m+\ell}{\ell + 1}(n-1) + \binom{m+\ell}{\ell +1} = \binom{m+\ell}{\ell + 1}\cdot n
\end{eqnarray*}
as wanted.
\end{proof}

Notice that the graded canonical module $\rmK_R$ of $R$ is given by the formula $\rmK_R = Q^{n-m}(-(t-1)m)$ (see \cite{BH2, BV}) and therefore we get the following corollary.

\begin{cor}\label{3.4}
Suppose that $t=2$. Then there is the equality
$$
\mu_R(M \rmK_R) = \binom{n}{n-m +1}\cdot n.
$$
\end{cor}

Let us note the following.

\begin{lem}\label{3.5}
There is the inequality
$$
(m^2 -2m + \ell +2)\cdot\binom{m+\ell-1}{\ell} > (\ell +1)\left(m^2+(\ell-2)m-(\ell-2)\right)
$$
for every $\ell \ge 1$, $m \ge 3$.
\end{lem}

\begin{proof}
By induction on $\ell$. It is obvious that the case where $\ell=1$. Suppose that $\ell >1$ and the assertion hold for $\ell -1$. Therefore by induction arguments, we have
$$
(m^2 -2m + \ell +1)\cdot\binom{m+\ell-2}{\ell-1} > \ell\cdot\left(m^2+(\ell-3)m-(\ell-3)\right). \quad \quad (**)
$$
To prove the inequality
$$
(m^2 -2m + \ell +2)\cdot\binom{m+\ell-1}{\ell} > (\ell +1)\left(m^2+(\ell-2)m-(\ell-2)\right) \quad \quad (*)
$$
it is enough to show that 
$$
(\text{LHS of }(*))-(\text{LHS of }(**)) > (\text{RHS of }(*))-(\text{RHS of }(**)). \quad \quad (\sharp)
$$
Indeed, we have 
\begin{eqnarray*}
&&(\text{LHS of }(\sharp))-(\text{RHS of }(\sharp)) \\
&=&  \binom{m+\ell -2}{\ell -1} + (m^2 -2m + \ell +2) \binom{m+\ell -2}{\ell} - (m^2 +(2\ell-2)m -(2\ell -2)) \\
&\ge& \binom{m+\ell -2}{\ell -1} + (m^2 -2m + \ell +2) \binom{\ell +1}{\ell} - (m^2 -2m +\ell +2 + (2m-3)\ell) \\ 
&=& \binom{m+\ell -2}{\ell -1} + \ell\cdot ((m-2)^2+ \ell + 1) >0
\end{eqnarray*}
where the second inequality follows from the assumption $m\ge3$.
\end{proof}

We are now in a position to prove Theorem \ref{3.2}.

\begin{proof}[Proof of Theorem \ref{3.2}]
We consider the positive integer $\ell = n-m$. Since $A=R_M$ is an almost Gorenstein local ring, there exists an exact sequence
$$
0 \to A \to \rmK_A \to C \to 0
$$
of $A$-modules such that $C$ is an Ulrich $A$-module, where $\m = MR_M$. Thanks to Corollary \ref{2.3}, we get
$$
0 \to \m \to \m \rmK_A \to \m C \to 0
$$
because $A$ is not a Gorenstein ring, so that 
\begin{eqnarray*}
\mu_A(\m \rmK_A) &\le& \mu_A(\m) + \mu_A(\m C) \\
&\le& m(m+\ell) + (d-1) (\rmr(A)-1)
\end{eqnarray*}
where $d = \dim A = 2m + \ell -1$ and $\rmr(A)$ denotes the Cohen-Macaulay type of $A$. Since $\rmK_A=(x_{i1}\mid 1 \le i \le m)^{\ell}$, we see that $\rmr(A) = \binom{m+\ell -1}{\ell}$. Therefore by Corollary \ref{3.4}
\begin{eqnarray*}
\binom{m+\ell}{\ell+1}\cdot (m+\ell) 
&\le& m(m+\ell) + (2m + \ell -2)\cdot\left(\binom{m+\ell-1}{\ell}-1\right) \\
&=& (m^2 + (\ell -2)m -(\ell -2)) + (2m + \ell -2)\cdot\binom{m+\ell-1}{\ell}
\end{eqnarray*}
which implies 
$$
\binom{m+\ell-1}{\ell+1}\cdot (m+\ell) =  (m^2 + (\ell -2)m -(\ell -2)) + (m-2) \cdot\binom{m+\ell -1}{\ell}
$$
by using the formula $\binom{m+\ell}{\ell+1} = \binom{m+\ell -1}{\ell + 1}+ \binom{m+\ell-1}{\ell}$.
Moreover since $\binom{m+\ell - 1}{\ell+1} = \binom{m+\ell -1}{\ell}\frac{m-1}{\ell + 1}$, we obtain the inequality
$$
(m^2 -2m + \ell +2)\cdot\binom{m+\ell-1}{\ell} \le (\ell +1)\left(m^2+(\ell-2)m-(\ell-2)\right)
$$
and hence $m=2$ by Lemma \ref{3.5}.
\end{proof}


\section{Survey on the resolution of determinantal rings}

In this section we deal with the resolution of determinantal rings in order to determine the lower bound of the number of generators for $M\rmK_R$, which plays an important role of the proof of Theorem \ref{main}. Let us now trace back to a history, if $t=1$, then Koszul complex of $x_{ij}$ gives the resolution of determinantal rings. In 1962, J. A. Eagon and D. G. Northcott (\cite{EN}) constructed such resolution in the case where the determinantal ideals ${\rm I}_t(X)$ are generated by the maximal minors of $X$. It is also known by T. H. Gulliksen and O. G. Neg\r{a}rd (\cite{GN}) when $t=m-1$, $m=n$. Finally, in 1978, A. Lascoux (\cite{L}) found the minimal free resolution of determinantal ring for arbtrary $t$, $m$, and $n$ if the base ring contains the field of rational numbers $\Bbb Q$. Besides this the resolution has been discovered in different ways by P. Roberts (\cite{R}). He gave more down-to-earth construction of the Lascoux's resolution, so we adopt his approach. In this section we will show a brief survey on the construction of the resolution given by P. Roberts. In particular, let us concentrate on how to get the ranks of the pieces of the resolution.

In what follows, let $t \ge 1$ and $m \ge n \ge 1$ be integers.
We fix a Noetherian local ring $(S, \n)$ which contains a field of rational numbers $\Bbb Q$. Let $F$, $G$ be free $S$-modules with $\rank_S F=m+t-1$ and $\rank_S G=n+t-1$, respectively. Let $\phi =(r_{ij}) : F \to G$ be a $S$-linear map such that $r_{ij} \in \n$. Let $\lambda(m, n)$ denote the Young tableau which consists of $n$ rows and $m$ squares, where the $i$-th row contains the numbers from $(i-1)m + 1$ to $im$ in increasing order, namely
$$\hspace{-3em}
\lambda(m, n) \ \ = \ \ 
 \ytableausetup{mathmode, boxsize=2.2em}
 \begin{ytableau}
  1 & 2 &  \none[\dots] & \scriptstyle m-1 & m \\
\scriptstyle  m+1 & \scriptstyle  m+2 &  \none[\dots] & \scriptstyle 2m-1 & 2m \\
  \none[\vdots] & \none[\vdots] & \none[\ddots] & \none[\vdots] & \none[\vdots] \\
   &  &  \none[\dots] &  & mn
 \end{ytableau}
$$

\vspace{1em}

Let $k$ be an integer such that $0 \le k \le mn$. Take a partition $\lambda = (\lambda_1, \lambda_2, \ldots, \lambda_n)$ of $k$ so that $\lambda_1 \ge \lambda_2 \ge \cdots \ge \lambda_n$, $\sum_{i=1}^n\lambda_i = k$ and $\lambda_i \le m$ for each $1 \le i \le n$.
Let $c_{\lambda}$ denote the Young symmetrizer, that is 
$$
c_{\lambda} = \left(\sum_{\sigma \in P_{\lambda}}\sigma \right)\cdot \left(\sum_{\sigma \in Q_{\lambda}}\operatorname{sgn}(\sigma) \sigma \right)
$$
in the group algebra $\Bbb Q[S_k]$ of the symmetric group $S_k$, where 
$$
P_{\lambda} = \{ \sigma \in S_k \mid \sigma \ \text{preserves each row}\} \ \text{and}  \ 
Q_{\lambda} = \{ \sigma \in S_k \mid \sigma \ \text{preserves each column}\}.
$$
Then thanks to \cite[Theorem IV. 3.1]{B}, there exists $k_{\lambda} \in \Bbb Q$ such that $k_{\lambda} \ne 0$ and $c_{\lambda}^2 = k_{\lambda}c_{\lambda}$. Let
$$
e(\lambda) = \frac{1}{k_{\lambda}}c_{\lambda}
$$
be the idempotent element of $\Bbb Q[S_k]$. Since $S$ contains the field of rational numbers $\Bbb Q$, the group algebra $\Bbb Q[S_k]$ acts on the tensor product $F^{\otimes k} = F \otimes_S F \otimes_S \cdots \otimes_S F$ of modules. Hence the image of homothety map by $e(\lambda)$ forms a free $S$-submodule of $F^{\otimes k}$ which is denoted by $e(\lambda)F$.

We will define the Young tableaux $\lambda_F$, $\lambda_G$ derived from $\lambda$ as follows. The $i$-th column of $\lambda_F$ consists of $\lambda_i$ squares which contain the numbers of the $(n-i +1)$-th row of $\lambda(m,n)$ in reverse order. Let $\lambda_G$ be the tableau derived from $\lambda(m, n)$ by removing the numbers of $\lambda_F$. Here we now associate to each square of $\lambda(m, n)$ either a square or a set of $t$ squares of $\lambda(m, n+t-1)$. The square in the $(i, j)$ position of $\lambda(m, n)$ corresponds to the following.
\begin{itemize}
\item[$(1)$] The square in the $(i, j)$ position of $\lambda(m, n+t-1)$ if $j-i > m-n$.
\item[$(2)$] The set of $t$ squares from the $(i, j)$ position to the $(i + t-1, j)$ position if $j-i=m-n$.
\item[$(3)$] The square in the $(i+t-1, j)$ position if $j-i < m-n$.
\end{itemize}

Let $\lambda_F(t)$ be the tableau defined by replacing each square of $\lambda_F$ by the corresponding square or the set of $t$ squares of $\lambda(m, n+t-1)$. We consider the tableaux $\lambda_G(t)$ which is obtained from $\lambda(m, n+t-1)$ by removing the squares of $\lambda_F(t)$.

\begin{defn}
We define
$$
C_k=C_k(t) = \sum_{|\lambda|= k}e(\lambda_F(t))F \otimes_S e(\lambda_G(t))G
$$
for every $0 \le k \le mn$, where $|\lambda| =k$ stands for the partition $\lambda = (\lambda_1, \lambda_2, \ldots, \lambda_n)$ of $k$ such that $\lambda_1 \ge \lambda_2 \ge \cdots \ge \lambda_n$, $\sum_{i=1}^n\lambda_i = k$ and $\lambda_i \le m$ for every $1 \le i \le n$.
\end{defn}

With the above notation the module $C_k$ and the suitable boundary maps (see \cite{R}) make a minimal $S$-free resolution of $S/{\rm I}_t(\phi)$:
$$
0 \to C_{mn} \to C_{mn-1} \to \cdots \to C_1 \to C_0 \to S/{\rm I}_t(\phi) \to 0.
$$

From now on we will compute the rank of the free module $C_k$. The strategy for the computation is the following. Firstly we determine all partitions $\lambda=(\lambda_1, \lambda_2, \ldots, \lambda_n)$ of $k$ with $\sum_{i=1}^n\lambda_i=k$ and $0 \le \lambda_i \le m$. We then have to find the corresponding Young diagrams $\lambda_F(t)$, $\lambda_G(t)$ and compute the ranks of free modules $e(\lambda_F(t))F$, $e(\lambda_G(t))G$.

As for the last step, we apply the following formula given by H. Boerner (\cite[Theorem VI. 1.3]{B}).  More precisely let $\lambda=(\lambda_1, \lambda_2, \ldots, \lambda_r)$ be a partition such that $\lambda_1 \ge \lambda_2 \ge \cdots \lambda_r$, $H$ a free $S$-module of rank $r \ge 0$. 
We denote by
$$
\Delta(x_1, x_2, \ldots, x_r) = \prod_{i<j} (x_i - x_j)
$$
the difference products of integers $x_1, x_2, \ldots, x_r$. Then the ranks of $e(\lambda)H$ is given by the formula
$$
\rank_S e(\lambda)H = \frac{\Delta(\ell_1, \ell_2, \ldots, \ell_r)}{\Delta(r-1, r-2, \ldots, 0)}
$$
where $\ell_i = \lambda_i + r -i$ for each $1 \le i \le r$.

The remainder of this section is devoted to compute the ranks of $C_{mn}$ and $C_{mn-1}$. Let us begin with the following, which corresponds to the Cohen-Macaulay type of the determinantal ring.

\begin{prop}\label{4.3}
There is the equality
$$
{\footnotesize
\rank_S C_{mn} = \frac{\displaystyle\prod_{j=0}^{m-n-1}\left(\prod_{i=0}^{n-1}(t+i+j)\right)1!\cdot 2! \cdots (n-2)!\cdot(n-1)!}{(m-n)!\cdot (m-n+1)! \cdots (m-2)! \cdot (m-1)!}}.
$$
\end{prop}

\begin{proof}
Let $\lambda = (\lambda_1, \lambda_2, \ldots, \lambda_n)$ be a partition of $mn$ such that $\lambda_1 \ge \lambda_2 \ge \cdots \ge \lambda_n$, $\sum_{i=1}^{n}\lambda_i = mn$, and $\lambda_i \le m$ for each $1 \le i \le n$. Then $\lambda_i = m$ for every $1 \le i \le n$. Notice that the Young tableau $\lambda_F$ has the form
$$
\lambda_F \ \ = \ \ 
 \ytableausetup{mathmode, boxsize=1.5em}
 \begin{ytableau}
*(black)   &  &  \none[\dots] &   \\
 & *(black)  &  \none[\dots] &  \\
  \none[\vdots] & \none[\vdots] & \none[\ddots] & \none[\vdots]  \\
   &  &  \none[\dots] &  *(black)  \\
      &  &  \none[\dots] &   \\
  \none[\vdots] & \none[\vdots] & \none & \none[\vdots]  \\ 
   &  &  \none[\dots] &   
 \end{ytableau}
$$
where the filled square corresponds to the string of $t$ squares in $\lambda(m, n+t-1)$. Therefore $\lambda_F(t)$ (resp. $\lambda_G(t)$) is the Young diagram which consists of $m+t-1$ (resp. $t-1$) rows and $n$ (resp. $m-n$) columns. Thus 
{\footnotesize
\begin{eqnarray*}
\rank_Se(\lambda_F(t))F &=& \frac{\Delta(m+t+n-2, m+t+n-3, \ldots, m+t-1, m+t-2, \ldots, n+1, n)}{\Delta(m+t-2, m+t-3, \ldots, n, n-1, \ldots, 1, 0)} \\
&=& 1
\end{eqnarray*}}
because $\rank_SF = m+t-1$. Hence $\rank_S C_{mn} = \rank_Se(\lambda_G(t))G$. 

On the other hand, since $\rank_S G=n+t-1$, we get
{\footnotesize
\begin{eqnarray*}
\rank_Se(\lambda_G(t))G &=& \frac{\Delta(m+t-2, m+t-3, \ldots, n+t-1, n+t-2, \ldots, m, n-1, \ldots, 1, 0)}{\Delta(n+t-2, n+t-3, \ldots, m, m-1, \ldots, n, n-1, \ldots, 1, 0)} \\
&=& \frac{\displaystyle\prod_{j=0}^{m-n-1}\left(\prod_{i=0}^{n-1}(t+i+j)\right)1!\cdot 2! \cdots (n-2)!\cdot(n-1)!}{(m-n)!\cdot (m-n+1)! \cdots (m-2)! \cdot (m-1)!}
\end{eqnarray*}}
as desired.
\end{proof}

Closing this section let us compute the rank of $C_{mn-1}$ in the case where $m\ne n$.

\begin{prop}\label{4.4}
Suppose that $m\ne n$. Then there is the equality
{\footnotesize
$$
\rank_S C_{mn-1} = \frac{\displaystyle\prod_{j=0}^{m-n-1}\left(\prod_{i=1}^{n-1}(t+i+j)\right)\prod_{i=0}^{m-n-2}(t+i)\left(t+m-1\right)1!\cdot 2! \cdots (n-2)!\cdot n!}{(m-n-1)!\cdot (m-n+1)!\cdot(m-n+2)! \cdots (m-2)! \cdot (m-1)!}. 
$$
}
\end{prop}

\begin{proof}
Similarly as in the proof of Proposition \ref{4.3},  let us choose a partition $\lambda = (\lambda_1, \lambda_2, \ldots, \lambda_n)$ of $mn-1$ such that $\lambda_1 \ge \lambda_2 \ge \cdots \ge \lambda_n$, $\sum_{i=1}^{n}\lambda_i = mn-1$, and $\lambda_i \le m$ for each $1 \le i \le n$. We then have $\lambda_i = m$ for every $1 \le i \le n-1$ and $\lambda_n=m-1$. Therefore since $m\ne n$, we get
$$
\lambda_F(t) \ \ = \ \ 
 \ytableausetup{mathmode, boxsize=1.5em}
 \begin{ytableau}
{}   & \none[\dots] &  \none[\dots] &   \\
 & \none[\dots] &  \none[\dots] &  \\
  \none[\vdots] & \none[\ddots] & \none & \none[\vdots]  \\ 
  &  \none[\dots] &   &   \\
   &  \none[\dots] &   
 \end{ytableau}\   \ \ \ \text{and} \ \ \
 \lambda_G(t) \ \ = \ \ 
 \ytableausetup{mathmode, boxsize=1.5em}
 \begin{ytableau}
{}   & \none[\dots] &  \none[\dots] &   \\
 & \none[\dots] &  \none[\dots] &  \\
  \none[\vdots] & \none[\ddots] & \none & \none[\vdots]  \\ 
  &   & \none[\dots]  &   \\
   &  \none &    \none
 \end{ytableau}
$$
where $\lambda_F(t)$ (resp. $\lambda_G(t)$) is the tableau which consists of $m+t-1$ (resp. $t$) rows and $n$ (resp. $m-n$) columns. Hence
{\footnotesize
\begin{eqnarray*}
\rank_Se(\lambda_F(t))F &=& \frac{\Delta(m+t+n-2, m+t+n-3, \ldots, m+t-1, m+t-2, \ldots, n+1, n-1)}{\Delta(m+t-2, m+t-3, \ldots, n+1, n, \ldots, 1, 0)} \\
&=& t+ m-1.
\end{eqnarray*}}
Moreover
{\footnotesize
\begin{eqnarray*}
\rank_Se(\lambda_G(t))G &=& \frac{\Delta(m+t-2, m+t-3, \ldots, n+t-1, n+t-2, \ldots, m, n, n-2, \ldots, 1, 0)}{\Delta(n+t-2, n+t-3, \ldots, m, m-1, \ldots, n, n-1, \ldots, 1, 0)} \\
&=& \frac{\displaystyle\prod_{j=0}^{m-n-1}\left(\prod_{i=1}^{n-1}(t+i+j)\right)\prod_{i=0}^{m-n-2}(t+i)\cdot1!\cdot 2! \cdots (n-2)!\cdot(n-1)!}{(m-n)!\cdot (m-n+1)! \cdots (m-2)! \cdot (m-1)!}
\end{eqnarray*}}
and hence we get the required equality.
\end{proof}


\section{Proof of Theorem \ref{main}}

In this section we maintain the notation as in Setting \ref{3.1}. We begin with the following lemma. 

\begin{lem}\label{5.1}
$R=S/{\rm I}_t(X)$ is an almost Gorenstein graded ring if and only if either $m=n$, or $m \ne n$ and $m=t=2$.
\end{lem}

\begin{proof}
The `if' part is due to \cite[Example 10.5]{GTT}. Let us make sure of the `only if' part. Suppose that $R$ is an almost Gorenstein graded ring and $m\ne n$. Remember that $\rma(R)= -(t-1)m$ and $\dim R=mn-(m-(t-1))(n-(t-1))$. Then thanks to \cite[Theorem 1.6]{GTT}, we have $\rma(R) = 1-\dim R$, which implies $1 = (t-1)(m-(t-1))$. Consequently $m=t=2$.
\end{proof}

For a moment, suppose that $k$ is a field of characteristic $0$. Look at the graded minimal $S$-free resolution
$$
0 \to F \to G \to \cdots \to S \to R \to 0 \ \ \ \ \ \ \  (\sharp)
$$
of the determinantal ring $R$. Then by Proposition \ref{4.4}, we obtain the equality

{\footnotesize
$$
\rank_S F = \frac{\displaystyle\prod_{j=0}^{n-m-1}\left(\prod_{i=0}^{m-t}(t+i+j)\right)1!\cdot 2! \cdots (m-t-1)!\cdot (m-t)!}{(n-m)!\cdot (n-m+1)! \cdots (n-t-1)! \cdot (n-t)!}.
$$}

\noindent
Moreover if $R$ is not a Gorenstein ring, then 

{\footnotesize
$$
\rank_S G = \frac{\displaystyle\prod_{j=0}^{n-m-1}\left(\prod_{i=1}^{m-t}(t+i+j)\right)\prod_{i=0}^{n-m-2}(t+i) \cdot n\cdot 1!\cdot 2! \cdots (m-t-1)!\cdot (m-t+1)!}{(n-m-1)!\cdot (n-m+1)!\cdot(n-m+2)! \cdots (n-t-1)! \cdot (n-t)!}.
$$
}

\noindent
We now take the $S(-mn)$-dual of the resolution $(\sharp)$ to get the presentation of the graded canonical module $\rmK_R$, which yields that
$$
\mu_R(M\rmK_R) \ge mn \cdot \rmr(R) -\rank_S G\ \ \ \ \ \ \  (\dagger)
$$
by using Proposition \ref{2.4} and \cite[Section 2]{R}.
Let us consider the rational number

{\footnotesize 
$$
\alpha = \frac{\displaystyle\prod_{j=0}^{n-m-1}\left(\prod_{i=1}^{m-t}(t+i+j)\right)\prod_{i=0}^{n-m-2}(t+i) \cdot 1!\cdot 2! \cdots (m-t-1)!\cdot (m-t)!}{(n-m-1)!\cdot (n-m+1)!\cdot(n-m+2)! \cdots (n-t-1)! \cdot (n-t)!}.
$$
}

We then have the following equalities
$$
\rmr(R) = \frac{t+n-m-1}{n-m}\cdot \alpha, \ \ \  \rank_S G = n\cdot (m-t+1)\cdot \alpha.
$$

\begin{rem}\label{chacteristic}
Notice that the Hilbert series of $R$ does not depend on the field, so is the Hilbert series of $\rmK_R$ (\cite[(6.2.3) Proposition]{Weyman}, \cite[4.4 {\sc Theorem}]{Stanley}). Since $R$ is homogeneous and level, we conclude that $\mu_R(M\rmK_R)$ does not depend on the characteristic of the field $k$. Hence the inequality $(\dagger)$ holds for any characteristic of $k$.
\end{rem}

We are now ready to prove Theorem \ref{main}.

\begin{proof}[Proof of Theorem \ref{main}]
The implication $(1) \Rightarrow (2)$ follows from the definition (see \cite[Section 8]{GTT}).
By Lemma \ref{5.1}, it suffices to show that $(2) \Rightarrow (3)$. We may assume $A=R_M$ is not a Gorenstein ring, that is $m\ne n$. By Remark \ref{chacteristic} we may also assume that $k$ is a field of characteristic $0$. Choose an exact sequence
$$
0 \to A \to \rmK_A \to C \to 0
$$
of $A$-modules such that $C \ne (0)$ is an Ulrich $A$-module.
Then since $m \ne n$, we obtain the exact sequence
$$
0 \to \m \to \m \rmK_A \to \m C \to 0
$$
of $A$-modules, which yields that
\begin{eqnarray*}
\mu_A(\m \rmK_A) &\le& \mu_A(\m) + \mu_A(\m C) \\
&\le& mn + (d-1)(\rmr(A) -1 )
\end{eqnarray*}
where $d= \dim A$ and $\m =MR_M$.
Therefore we have the following inequalities
$$
mn \cdot \rmr(A) -\rank_S G \le mn + (d-1)(\rmr(A) -1 )
$$
so that
$$
\left(mn-(d-1)\right)(\rmr(A) -1 )\le \rank_S G.
$$
Hence 
$$
\{(m-(t-1))(n-(t-1))+1\}\left(\frac{t+n-m-1}{n-m}\cdot\alpha -1 \right)\le n \cdot (m-(t-1))\alpha.
$$
Let $\ell = n-m~ (>0)$ and $s = m-t ~(\ge 0)$. The above inequality implies the following
$$
\{(s+1)(s+\ell +1)+1\}\{ (t+\ell -1)\alpha -\ell \} \le (s + t+ \ell)\cdot \ell \cdot(s+1)\cdot\alpha.
$$
Then a direct computation shows
$$
\alpha \cdot\{(s+1)^2(t-1) + (t+\ell -1)\} \le \ell \cdot\{(s+1)(s+\ell +1)+1\}.
$$
We put the rational number
{\footnotesize
\begin{eqnarray*}{\hspace{1em}}
\beta &=& \frac{1}{s+1}\binom{s+t}{s} \cdot\frac{1}{s+2}\binom{s+t+1}{s+1}\cdot\frac{1}{(s+3)(t+1)}\binom{s+t+2}{s+2} \\
&& \\
&& \cdots  \frac{1}{(s+\ell -1)(t+1)\cdots(t+\ell -3)}\binom{s+t+\ell -2}{s+ \ell -2}\cdot\frac{t}{(s+\ell)(t+1)\cdots(t+\ell -1)}\binom{s+t+\ell -1}{s+ \ell -1}
\end{eqnarray*}}

\noindent
whence
$$
\alpha =  1!\cdot 2! \cdots (\ell-2)!\cdot \ell !\cdot\beta. 
$$ 
Therefore we get
$$
\left(\prod_{i=1}^{\ell -1}i !\right) \cdot \beta \cdot \{(s+1)^2(t-1) + (t+\ell -1)\} \le  (s+1)(s+\ell +1)+1.
$$
Notice that 
{\footnotesize
\begin{eqnarray*}
\beta \cdot \left(\prod_{i=1}^{\ell}(s+i)! \right)&=& (t+\ell -1)^{\ell -1}\left(\prod_{i=0}^{\ell-2} (t+i)^{i+1}(t+s+1 +i)^{\ell -(i + 1)}\right) \left(\prod_{i=\ell}^s (t+i)^{\ell}\right) 
\end{eqnarray*}}

\noindent 
and hence we get the inequality
{\footnotesize
\begin{eqnarray*}
&&\left(\prod_{i=1}^{\ell -1}i !\right) (t+\ell -1)^{\ell -1}\left(\prod_{i=0}^{\ell-2} (t+i)^{i+1}(t+s+1 +i)^{\ell -(i + 1)}\right) \left(\prod_{i=\ell}^s (t+i)^{\ell}\right) \{(s+1)^2(t-1) + (t+ \ell -1)\} \\
&& \\ 
&&{\hspace{0.5em}}\le \left(\prod_{i=1}^{\ell}(s+i)! \right)\{(s+1)(s+\ell +1)+1\}.
\end{eqnarray*}}
We now assume that $t \ge 3$ to make a contradiction. By the above inequality, we have
$$
(s+\ell + 2)(s+2)\cdot \left(\prod_{i=3}^{\ell + 1}(s+i)^2\right)\cdot \{2(s+1)^2 + \ell +2\} \le \ell ! \cdot (\ell +2)! \cdot \{(s+1)(s+\ell +1)+1\}
$$
so that 
\begin{eqnarray*}
&&(s+1)(s+\ell + 1) \left\{2(s+1)(s+2)\left(\prod_{i=3}^{\ell}(s+i)^2\right) (s+\ell +1)(s+\ell +2) - \ell! \cdot (\ell +2)!\right\} \\
&& + (\ell + 2)(s+2) \left(\prod_{i=3}^{\ell+1}(s+i)^2 \right)(s+ \ell +2) - \ell ! \cdot(\ell +2)! \le 0.
\end{eqnarray*}
Since $s \ge 0$, the left hand side of the above inequality is not less than
$$
\frac{1}{2}\cdot \ell ! \cdot (\ell + 2)! \cdot (\ell + 1)(\ell + 2) - \ell ! \cdot (\ell + 2)!
$$
which is a positive integer. This is a contradiction and hence $t=2$. Therefore by Theorem \ref{3.2}, we finally get $m=2$. This completes the proof. 
\end{proof}

Let us note the following, where $k[[X]]$ denotes the formal power series ring over the field $k$.

\begin{cor}\label{5.2}
$k[[X]]/{\rm I}_t(X)$ is an almost Gorenstein local ring if and only if either $m=n$, or $m\ne n$ and $m=t=2$.
\end{cor}

\begin{proof}
See \cite[Theorem 3.9]{GTT}.
\end{proof}

\begin{rem}\label{5.3}
The condition $m\ne n$ and $m=t=2$ is equivalent to saying that the local ring $k[[X]]/{\rm I}_t(X)$ has a minimal multiplicity (e.g., \cite[Section 10]{GTT}).
\end{rem}


\vspace{0.5cm}

\begin{ac}
The author would like to thank Kazuhiko Kurano and Mitsuyasu Hashimoto for their valuable advices and comments.
\end{ac}


\end{document}